\documentclass[a4paper,12pt]{article}

\usepackage{color}
\usepackage[pdftex]{graphicx}
\usepackage{amsmath,amssymb,amsthm}
\usepackage{wrapfig}
\usepackage{ifpdf}

\newtheorem{theorem}{Theorem}[section]
\newtheorem{cor}{Corollary}[section]
\newtheorem{lemma}{Lemma}[section]

\newcommand{\conv}{\mathop{\rm conv}\nolimits}

\newcommand{\ns}{\mathop{\rm SN}\nolimits}

\DeclareMathOperator{\dg2}{deg_2}

\title {Generalizations of Tucker--Fan--Shashkin lemmas}

\author {Oleg R. Musin\thanks{This research is partially supported by the NSF grant DMS-1400876 and the RFBR grant 15-01-99563.}}

\begin{document}

	\ifpdf \DeclareGraphicsExtensions{.pdf, .jpg, .tif, .mps} \else
	\DeclareGraphicsExtensions{.eps, .jpg, .mps} \fi	
	
\date{}
\maketitle

\begin{abstract} Tucker and Ky Fan's lemma are combinatorial analogs of the Borsuk--Ulam theorem (BUT). In 1996, Yu. A. Shashkin proved a version of Fan's lemma, which is a combinatorial analog of the odd mapping theorem (OMT).  We consider generalizations of these lemmas for BUT--manifolds, i.e. for manifolds that satisfy BUT. Proofs rely on a generalization of the OMT and on a lemma about the doubling of manifolds with boundaries that are BUT--manifolds.
\end{abstract}

\medskip

\noindent {\bf Keywords:} Tucker lemma, Ky Fan lemma, Shashkin lemma, Borsuk--Ulam theorem, degree of mapping. 

\section{Tucker's, Fan's and Shashkin's lemmas}

Throughout this paper the symbol ${\mathbb R}^d$ denotes the Euclidean space of dimension $d$.  We denote by  ${\mathbb B}^d$ the $d$-dimensional unit ball and by  ${\mathbb S}^d$ the $d$-dimensional unit sphere. If we consider ${\mathbb S}^d$ as the set of unit vectors $x$ in ${\mathbb R}^{d+1}$, then points $x$ and $-x$ are called {\it antipodal} and the symmetry given by the mapping
 $x \to -x$ is called the {\it antipodality} on  ${\mathbb S}^d$.

\subsection{Tucker and Fan's lemma}

Let $T$ be a triangulation of the $d$-dimensional ball ${\mathbb B}^d$. We call $T$ {\it antipodally symmetric on the boundary}  if the set of simplices of $T$ contained in the boundary of  ${\mathbb B}^d = {\mathbb S}^{d-1}$ is an antipodally symmetric triangulation of  ${\mathbb S}^{d-1}$; that is if $s\subset {\mathbb S}^{d-1}$ is a simplex of $T$, then $-s$ is also a simplex of $T$.

\medskip

\noindent {\bf Tucker's lemma \cite{Tucker}} {\it Let $T$ be a triangulation of  ${\mathbb B}^d$ that is antipodally symmetric on the boundary. Let $$L:V(T)\to \{+1,-1,+2,-2,\ldots, +d,-d\}$$ be a  labelling of the vertices of $T$ that is antipodal (i. e. $L(-v)=-L(v)$)  for every vertex $v$ on the boundary. Then there exists an edge in $T$ that is {\bf complementary}, i.e. its two vertices are labelled by opposite numbers.}

\medskip

\begin{figure}
\begin{center}

  \includegraphics[clip,scale=0.7]{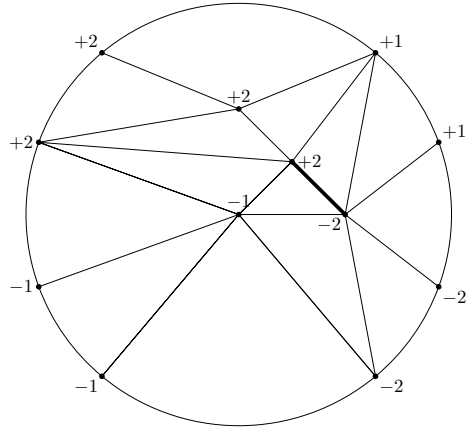}
\end{center}
\caption{Illustration of Tucker's lemma}
\end{figure}

\medskip

There is also a version of Tucker's lemma for spheres:

\medskip

\noindent  {\bf Spherical Tucker's lemma.}  {\it Let $T$ be a centrally symmetric triangulation of the sphere ${\Bbb S}^d$. Let $$L:V(T)\to \{+1,-1,+2,-2,\ldots, +d,-d\}$$ be an antipodal labelling. Then there exists a complementary edge.}

\medskip

Tucker's lemma was extended by Ky Fan \cite{KyFan}:

\medskip

\noindent {\bf Ky Fan's lemma.}  {\it Let $T$ be a centrally symmetric triangulation of the sphere ${\Bbb S}^d$. Suppose that each vertex $v$ of $T$ is assigned a label $L(v)$ from $\{\pm1,\pm2,\ldots,\pm n\}$ in such a way that $L(-v)=-L(v)$. Suppose this labelling does not have complementary edges. Then there are an odd number of $d$-simplices of $T$ whose labels are of the form $\{k_0,-k_1,k_2,\ldots,(-1)^dk_d\}$, where $1\le k_0<k_1<\ldots<k_d\le n$. In particular, $n\ge d+1$.}

\subsection{Shashkin's lemma}
In the 1990's, Yu. A. Shashkin published several works related to discrete versions of classic fixed point theorems  \cite{Shashkin, Shashkin1, Shashkin2, ShashkinT, Shashkin99}. In  \cite{ShashkinT} he proved the following theorem:

\medskip

\noindent {\bf Shashkin's lemma.}  {\it Let $T$ be a triangulation of a planar polygon that is antipodally symmetric on the boundary. Let $$L:V(T)\to \{+1,-1,+2,-2, +3,-3\}$$ be a labelling of the vertices of $T$ that satisfies $L(-v)=-L(v)$ for every vertex $v$ on the boundary. Suppose that this labelling does not have complementary edges. Then for any numbers $a,b,c$, where $|a|=1,\; |b|=2, \; |c|=3$, the total number of triangles in $T$  with labels $(a,b,c)$ and $(-a,-b,-c)$ is odd.}

\medskip

%---------------Fig 2------------------------------------------------------------
\begin{center}
\begin{picture}(320,160)(-140,-80)
% Fig.
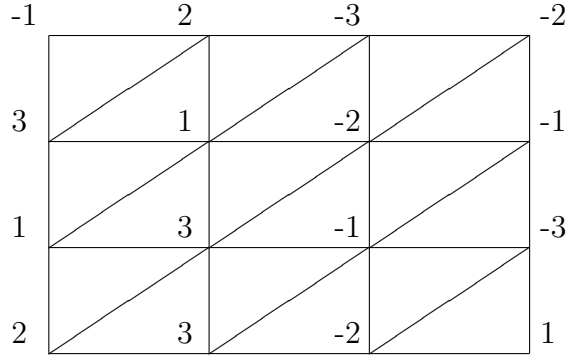
\put(-90,-70){Figure 2: Illustration of Shashkin's lemma.}

\multiput(-70,-40)(60,0){4}%
{\line(0,1){120}}

\multiput(-70,-40)(0,40){4}%
{\line(1,0){180}}

\put(-70,40){\line(3,2){60}}
\put(-70,0){\line(3,2){120}}
\put(-70,-40){\line(3,2){180}}

\put(50,-40){\line(3,2){60}}
\put(-10,-40){\line(3,2){120}}

\put(-84,-36){2}
\put(-84,4){1}
\put(-84,44){3}
\put(-84,84){-1}

\put(-22,-36){3}
\put(-22,4){3}
\put(-22,44){1}
\put(-22,84){2}

\put(37,-36){-2}
\put(37,4){-1}
\put(37,44){-2}
\put(37,84){-3}

\put(114,-36){1}
\put(114,4){-3}
\put(114,44){-1}
\put(114,84){-2}

\end{picture}
\end{center}

\medskip

\noindent{\bf Remark.} In other words, Shashkin proved that if $(a,b,c)=(1,2,3), \, (1,-2,3), \, (1,2,-3)$ and $(1,-2,-3)$, then the number of triangles with labels $(a,b,c)$ or $(-a,-b,-c)$ is odd. Denote this number by $\ns(a,b,c)$. Then in Fig. 2 we have
$$ \ns(1,2,3)=3, \; \ns(1,-2,3)=1,\; \ns(1,2,-3)=3, \; \ns(1,-2,-3)=3.$$

Note that this result was published only in Russian and only for two--dimensional case. Moreover, Shashkin attributes this theorem to Ky Fan \cite{KyFan}.

Actually, Shashkin's lemma can be derived from Ky Fan's lemma for $n=d+1$.  However. Shashkin's proof is different and relies on the odd mapping theorem (OMT).  In fact, this lemma is a discrete version of the OMT. That is why we distinguish this result as {\it Shashkin's lemma}.

\medskip

The following is a spherical version of Shashkin's lemma.

\medskip

\noindent {\bf Spherical Shashkin's lemma.}   {\it Let $T$ be a centrally symmetric  triangulation of \, ${\Bbb S}^d$. Let
$$
L:V(T)\to \Pi_{d+1}:=\{+1,-1,+2,-2,\ldots, +(d+1),-(d+1)\}
$$
be an antipodal  labelling of $T$.  Suppose that this labelling does not have complementary edges. Then for any set of  labels $\Lambda:=\{\ell_1,\ell_2,\ldots,\ell_{d+1}\}\subset\Pi_{d+1}$ with $|\ell_i|=i$ for all $i$, the number of $d$--simplices in $T$ that are labelled by $\Lambda$  is odd. }

\subsection{Main results}

In \cite{Mus} we invented BUT (Borsuk--Ulam Type) -- manifolds.
Theorems 3.1--3.4 in this paper extend Tucker's and Shashkin's lemmas for BUT--manifolds. Namely, Theorem 3.1 and Theorem 3.2 are extensions of the spherical Tucker and Shahskin lemmas,  where ${\Bbb S}^d$ is substituted by a BUT--manifold.  Theorems 3.3 and 3.4 are extensions  of the original Tucker and Shashkin lemmas, where   ${\Bbb B}^d$ is substituted by a manifold $M$ with boundary $\partial M$ that is a BUT--manifold.

 Our proof of Theorem 3.2  is relies on a generalization of the odd mapping theorem for BUT--manifolds:

\medskip

\noindent {\bf Theorem 2.1.} {\it Let $(M_1,A_1)$ and $(M_2,A_2)$ be  BUT--manifolds. Then any odd  continuous mapping $f:M_1\to M_2$ has odd degree.}

\medskip

Theorems 3.3 and 3.4 follow from Theorems 3.1 and 3.2 by using Lemma 3.1, which is about the doubling of manifolds with boundaries that are BUT--manifolds.

In Section 4 we extend  for BUT--manifolds Shaskin's proof of two Tucker's theorems about covering families from \cite{Tucker}. Actually, these theorems are corollaries of Theorem 3.2.

\section{The odd mapping theorem}

We say that a mapping $f:{\Bbb S}^d\to {\Bbb S}^d$ is {\it odd} or {\em antipodal} if $f(-x)=-f(x)$ for all $x\in {\Bbb S}^d$. If $f$ is a continuous mapping, then  $\deg{f}$ (the degree of $f$) is well defined.

Let $f:M_1\to M_2$ be a continuous map between two closed manifolds $M_1$ and $M_2$ of the same dimension. The degree  is a number that represents the amount of times that the domain manifold wraps around the range manifold under the mapping. Then $\dg2(f)$ (the degree modulo 2) is 1 if this number is odd and 0 otherwise. It is well known that $\dg2(f)$ of a continuous map $f$  is a homotopy invariant (see \cite{Milnor}).

The classical {\bf odd mapping theorem} states that

\medskip

\noindent {\it Every continuous odd mapping $f:{\Bbb S}^d\to {\Bbb S}^d$  has odd degree.}

\medskip

Shashkin \cite{ShashkinT} (see also \cite[Proposition 2.4.1]{Mat}) gives a proof of this theorem for simplicial mappings $f:{\Bbb S}^d\to {\Bbb S}^d$. Conner and Floyd \cite{CF60} considered Theorem 2.1 for a wide class of spaces. Here we extend the odd mapping theorem for BUT--manifolds. In our paper \cite{Mus}, we extended the Borsuk--Ulam theorem for manifolds.\\
Let  $M$ be a connected compact PL (piece-wise linear) $d$-dimensional manifold without boundary with a free simplicial involution $A:M\to M$, i. e. $A^2(x)=A(A(x))=x$ and $A(x)\ne x$.
We say that a pair $(M,A)$ is a {\it BUT (Borsuk-Ulam Type) manifold} if for any continuous  $g:M \to {\Bbb R}^d$ there is a point $x\in M$ such that $g(A(x))=g(x)$. Equivalently, if a continuous  map $f:M \to {\Bbb R}^d$  is { antipodal}, i.e. $f(A(x))=-f(x)$,  then the set of zeros $Z_f:=f^{-1}(0)$ is not empty.

In \cite{Mus}, we found several equivalent necessary and sufficient conditions for manifolds to be BUT. In particular,

\medskip

\noindent {\it $M$ is a $d$--dimensional BUT--manifold if and only if $M$ admits an antipodal continuous transversal to zeros mapping $h:M \to {\Bbb R}^d$ with $|Z_h|=2\pmod{4}$.}

\medskip

\noindent A continuous mapping $h:M \to {\Bbb R}^d$ is called {\it transversal to zero} if there is an open set $U$ in ${\Bbb R}^d$ such that $U$ contains $0$, $U$ is homeomorphic to the open $d$-ball and $h^{-1}(U)$ consists of a finite number open sets in $M$ that are homeomorphic to open $d$-balls.

The class of BUT--manifolds is sufficiently large. It is clear that $({\Bbb S}^d,A)$ with $A(x)=-x$ is a BUT-manifold.   Suppose that $M$ can be represented as a connected sum $N\# N$, where $N$ is a closed PL manifold. Then $M$ admits a free involution. Indeed, $M$ can be  ``centrally symmetrically" embedded to ${\Bbb R}^k$, for some $k$, and the antipodal symmetry $x\to -x$ in ${\Bbb R}^k$ implies a free involution $A:M\to M$ \cite[Corollary 1]{Mus}. For instance,   orientable  two-dimensional  manifolds $M^2_g$ with even genus $g$ and non-orientable manifolds $P^2_m$ with even $m$, where $m$ is the number of  M\"obius bands,  are BUT-manifolds.

Let $M_i,\; i=1,2$, be a manifold with a free involution $A_i$. We say that a mapping $f:M_1\to M_2$ is {\em antipodal} (or {\em odd}, or {\em equivariant}) if $f(A_1(x))=A_2(f(x))$ for all $x\in M_1$.

\begin{theorem} Let $(M_1,A_1)$ and $(M_2,A_2)$ be $d$-dimensional BUT--manifolds. Then any odd  continuous mapping $f:M_1\to M_2$ has odd degree.
\begin{proof} Since $(M_2,A_2)$ is BUT, there is a continuous antipodal transversal to zeros mapping $g:M_2 \to {\Bbb R}^d$ with $|Z_g|=4m_2+2$ \cite[Theorem 2]{Mus}.
	
Let $h:=g\circ f$. Then $h:M_1 \to {\Bbb R}^d$ is continuous and antipodal. 	Since the degree of a mapping is a homotopy invariant, without loss of generality we may assume that $h$ is a transversal to zero mapping (see  \cite[Lemma 3]{Mus}). Therefore $|Z_h|=4m_1+2$.
On the other hand,
$$
|Z_h|=\sum\limits_{x\in Z_g}{|f^{-1}(x)|}.
$$
Then
$$
2m_1+1=(2m_2+1)\dg2{f}\pmod{2}.
$$
 Thus, the degree of ${f}$ is odd.
\end{proof}

\end{theorem}

\section{Tucker's and Shashkin's lemmas for BUT--manifolds}

In our papers \cite{Mus,MusSpT,MusVo} are considered extensions of Tucker's lemma. Here we consider generalizations of Tucker's and Shashkin's lemmas for manifolds with and without boundaries.

Let $T$ be an {\it antipodally symmetric} (or {\it antipodal}) triangulation of a BUT--manifold $(M,A)$. This means that $A:T\to T$ sends simplices to simplices. Denote by $\Pi_n$ the set of labels $\{+1,-1,+2,-2,\ldots, +n,-n\}$   and let $L:V(T)\to \Pi_n$ be a labeling of $T$. We say that this labelling is {\it antipodal} if $L(A(v))=-L(v)$. An edge $uv$ in $T$ is called {\it complementary} if $L(u)=-L(v)$.

\begin{theorem}{\bf (\cite[Theorem 4.1]{MusSpT})} \label{TBUT} Let $(M,A)$ be a $d$-dimensional BUT--manifold. Let $T$ be an antipodal triangulation of   $M$.   Then for any antipodal  labelling $L:V(T)\to \Pi_d$  there exists a complementary edge.
\end{theorem}

Any antipodal labelling $L:V(T)\to \Pi_n$ of an antipodally symmetric triangulation $T$ of $M$  defines a simplicial map $f_L:T\to {\Bbb R}^n$. Let $\{e_1,-e_1,e_2,-e_2,\ldots,e_n,-e_n\}$ be the standard orthonormal basis in ${\Bbb R}^n$.
For $v\in V(T)$, set $f_L(v):=e_i$ if $L(v)=i$ and  $f_L(v):=-e_i$ if $L(v)=-i$. Since $f_L$ is defined on $V(T)$,  it defines a simplicial  mapping $f_L:T\to  {\Bbb R}^n$ (See details in \cite[Sec. 2.3]{Mat}.)

The following theorem is a version of Shashkin's lemma for manifolds without boundary.

\begin{theorem} \label{SBUT} Let $(M,A)$ be a $d$-dimensional BUT--manifold. Let $T$ be an antipodally symmetric triangulation of   $M$.   Let $L:V(T)\to \Pi_{d+1}$  be an antipodal  labelling of $T$.  Suppose that this labelling does not have complementary edges. Then for any set of  labels $\Lambda:=\{\ell_1,\ell_2,\ldots,\ell_{d+1}\}\subset\Pi_{d+1}$ with $|\ell_i|=i$ for all $i$, the number of $d$--simplices in $T$ that are labelled by $\Lambda$  is odd.
\end{theorem}
\begin{proof} Since $L$ has no complimentary edges, $f_L:T\to  {\Bbb R}^{d+1}$ is an antipodal mapping of $M$ to the boundary of the crosspolytope $C^{d+1}$ that is the  convex hull
$
\conv{\{e_1,-e_1,\ldots,e_{d+1},-e_{d+1}\}}.
$
Note that $\partial C^{d+1}$ is a simplicial sphere ${\Bbb S}^d$, which is a BUT-manifold. Therefore, 
Theorem 2.2 implies that the number of preimages of the simplex in  $\partial C^{d+1}$ with indexes from $\Lambda$ is odd. It completes the proof.
\end{proof}

\noindent{\bf Remark.} Theorem 3.1 can be proved using the same arguments. Indeed, suppose that $L:V(T)\to \Pi_d$  has no complementary edges. Then $f_L$ sends $M$ to $\partial C^d$. Since $\dim{\partial C^d}=d-1$, $\deg{f_L}=0$. This contradicts Theorem 2.1.

\medskip

Now we extend Tucker's and Shashkin's lemmas for the case when $M$ is a manifold with boundary that is a BUT--manifold.  But first, prove that there exists a ``double'' of $M$ that is a BUT-manifold.

\begin{lemma} \label{Lemma1}  Let $M$ be a compact PL manifold with boundary $\partial M$. Suppose  $(\partial M,A)$ is a BUT--manifold. Then there is a BUT--manifold $(\tilde M,\tilde A)$ and a submanifold $N$ in $\tilde M$  such that  $N\simeq M$, $\tilde A|_{\partial N}\simeq A$, $ (N\setminus\partial N) \cap \tilde A(N\setminus\partial N)=\emptyset$ and
$$\tilde M\simeq (N\setminus\partial N)\cup \partial N \cup \tilde A(N\setminus\partial N).$$
\end{lemma}

\begin{proof} {\bf 1.} First we prove the following statement:\\
{\it
 Let $X$ be a finite simplicial complex. Let $Y$ be a subcomplex of $X$ with a free involution $A:Y\to Y$. Then there is a simplicial embedding $F$ of $X$ into 
${\Bbb R}^q_+:=\{(x_1,\ldots,x_q)\in{\Bbb R}^q: x_1\ge 0\}$, where $q$ is sufficiently large,  such that $Y$ is centrally symmetrically embedded in ${\Bbb R}^q$, i.e. $F(A(y))=-F(y)$ for all $y\in Y$, and $X\setminus Y$ is mapped into the interior of  ${\Bbb R}^q_+$.}
	
Indeed, let $v_1,v_{-1},\ldots,v_m,v_{-m}$ denote vertices of $Y$ such that $A(v_k)=v_{-k}$. Let $\{v_{m+1},\ldots, v_n\}$ be the set of vertices of $X\setminus Y$.

Denote by $C^n$ the $n$--dimensional crosspolytope that is the boundary of convex hull
$$
\conv{\{e_1,-e_1,\ldots,e_n,-e_n\}}
$$
of the vectors of the standard orthonormal basis and their negatives.

Now define an embedding $F:X\to C^n$. Let $F(v_k):=e_k$, $F(v_{-k}):=e_{-k}$, where $1\le k\le m$, $F(v_k):=e_k$, and $k=m+1,\ldots,n$.
Since $F$ is defined for all of the vertices of $X$,  it uniquely defines a simplicial (piecewise linear) mapping $F:X\to C^n\subset {\Bbb R}^n$.
Then
$$F(Y)\subset C^m\subset{\Bbb R}^m=\{(x_1,\ldots,x_n)\in {\Bbb R}^n: x_i=0, \; i=m+1,\ldots,n \},$$
$F(A(y))=-F(y)$ for all $y\in Y$  and
$$
F(X\setminus Y)\subset {\Bbb R}^{n-1}_+:=\{(x_1,\ldots,x_n)\in {\Bbb R}^n: x_{m+1}+\ldots+x_n>0\},
$$
as required.

\medskip
	
\noindent {\bf 2.}	Let $X=M$ and $Y=\partial M$. Then it follows from {\bf 1} that there is an embedding  $F:M\to{\Bbb R}^q_+$ with $F(\partial M)\subset {\Bbb R}^q$ and $F(A(y))=-F(y)$ for all $y\in \partial M$, where $q=n-1$.
Let
$$
\tilde M:=F(M)\cup(-F(M))\subset{\Bbb R}^{q+1}={\Bbb R}^q_+\cup (-{\Bbb R}^q_+) \; \mbox{ and } \; \tilde A(x):=-x \mbox{ for all } x\in\tilde M.
$$
	 It is clear that $\tilde M\simeq (N\setminus\partial N)\cup \partial N \cup \tilde A(N\setminus\partial N),$ where $N:=F(M).$
	
\medskip

\noindent {\bf 3.}	Let us prove that $(\tilde M,\tilde A)$ is BUT. Indeed, since $(\partial M,A)$ is BUT, there is a continuous antipodal transversal to zeros mapping $g:\partial M\simeq\partial N \to {\Bbb R}^{d-1}$ with $|Z_g|=4m+2$, where $d:=\dim{M}.$ We extend this mapping to  $h:\tilde M \to {\Bbb R}^{d}$  with $h|_{\partial N}=g$ and $|Z_h|=|Z_g|=4m+2$.

Let $v=(x_1,\ldots,x_{n})\in{\Bbb R}^{n}$ be a vertex of $\tilde M$. If $v\in\partial N$, then put
$$
h(v):=(g(v),0)\in {\Bbb R}^{d}.
$$
For $v\in\tilde M\setminus\partial N$ define
$$
h(v):=(0,\ldots,0,x_{m+1}+\ldots+x_n)\in{\Bbb R}^{d}.
$$
Then $h:\tilde M\to{\Bbb R}^{d}$ is an antipodal transversal to zeros mapping and $h^{-1}(0)=g^{-1}(0)$.
\end{proof}

\begin{theorem}
	Let $M$ be a $d$--dimensional compact PL manifold with boundary $\partial M$. Suppose  $(\partial M,A)$ is a BUT--manifold.
Let $T$ be a triangulation of  $M$ that antipodally symmetric on $\partial M$. Let $L:V(T)\to\Pi_d$ be a   labelling of  $T$ that is antipodal on the boundary. Then there is a complementary edge in $T$.
\end{theorem}	

\begin{theorem}
	Let $M$ be a $d$--dimensional  compact PL manifold with boundary $\partial M$. Suppose  $(\partial M,A)$ is a BUT--manifold.
Let $T$ be a triangulation of  $M$ that antipodally symmetric on $\partial M$.  Let $L:V(T)\to\Pi_{d+1}$ be a   labelling of  $T$ that is antipodal on the boundary and has no complementary edges. Then for any set of  labels $\Lambda:=\{\ell_1,\ell_2,\ldots,\ell_{d+1}\}\subset\Pi_{d+1}$ with $|\ell_i|=i$ for all $i$, the number of $d$--simplices in $T$ that are labelled by $\Lambda$ or $(-\Lambda)$  is odd.
\end{theorem}
\begin{proof} By Lemma \ref{Lemma1} there is a BUT--manifold $(\tilde M,\tilde A)$ that is the double of $M$.  We can extend  $T$ and $L$ from $M$ to an antipodal triangulation $\tilde T:=T\cup\tilde A(T)$ of $\tilde M$ and an antipodal labelling $\tilde L:V(\tilde T)\to \Pi_n$, where $n=d$ in Theorem 3.3 and $n=d+1$ in Theorem 3.4, such that $\tilde L|_T=L$.
	
Thus, for the case $n=d$ Theorem 3.3 follows from Theorem 3.1 and for $n=d+1$ Theorem 3.2 yields Theorem 3.4.
\end{proof}

\section{Shashkin's proof of Tucker's theorems}

In this section we consider two Tucker's theorems about covering families. 
Note that Tucker \cite{Tucker} obtained these theorem only for \, ${\Bbb S}^2$. Bacon \cite{Bacon} proved that statements in Theorems \ref{t41} and \ref{t42} are equivalent to the Borsuk--Ulam theorem for normal topological spaces $X$ with free continuous involutions $A:X\to X$.  (See also Theorem 2.1 in our paper \cite{MusVo}.)  Actually, these theorems  can be proved from properties of Schwarz's genus \cite{Sv} or Yang's cohomological index  \cite{Kar,MusVo}.

For the two--dimensional case  in the book \cite{Shashkin99}  Shashkin derives Tucker's theorems from his lemma. Here we extend his proof for BUT--manifolds of all dimensions.

\begin{theorem} \label{t41} Let $(M,A)$ be a $d$-dimensional BUT--manifold. Consider a family of closed sets $\{B_i, B_{-i}\},\, i=1,\ldots, d+1$, where $B_{-i}:=A(B_i)$, is such that $B_i\cap B_{-i}=\emptyset$ for all $i$.  If this family  covers $M$, then for any set of  indices $\{k_1,k_2,\ldots,k_{d+1}\}\subset\Pi_{d+1}$ with $|k_i|=i$ for all $i$, the intersection of all $B_{k_i}$  is not empty.
\end{theorem}
\begin{proof} Note that any PL manifold admits a metric. For a triangulation $T$ of $M$, the norm of $T$, denoted by $|T|$, is the diameter of the largest simplex in $T$.

Let $T_1,T_2,\ldots$ be a sequence of antipodal triangulations of $M$ such that $|T_i|\to0$.
Now for all $i$ define an antipodal labelling $L_i:V(T_i)\to \Pi_{d+1}$. For every  $v\in V(T_i)\subset M$ set
$$
L_i(v):=\ell, \mbox{ where } v\in B_\ell \mbox{ and } |\ell|=\min{\{|k|: v\in B_k\}}.
$$
Then $L_i$ satisfies the assumptions in Theorem \ref{SBUT} and $T_i$ contains a simplex $s_i$ with labels $\{k_1,k_2,\ldots,k_{d+1}\}\subset\Pi_{d+1}$.

 Since $M$ is compact and $|s_i|\to0$, the sequence $\{s_i\}$ contains a converging subsequence $P$  with limit $w\in M$. Then for $s_i\in P$ we have $V(s_i)\to w$.

By assumption, all $B_k$ are closed sets.  Therefore $w\in B_{k_j}$ for all $j=1,\ldots, d+1$, and thus $w\in \cap_j{B_{k_j}}$.
\end{proof}

\begin{theorem} \label{t42} Let $(M,A)$ be a $d$-dimensional BUT--manifold. Suppose that $M$ is covered by a family $\mathcal F$ of $d+2$ closed subsets $C_1,\ldots,C_{d+2}$. Suppose that all $C_i$ have no antipodal pairs $(x,A(x))$, in other words, $C_i\cap A(C_i)=\emptyset$. Let $0<k<d+2$. Then  any $k$ subsets from $\mathcal F$ intersect and there is a point $x$ in this intersection such that $A(x)$ belongs to the intersection of the remaining $(d+2-k)$ subsets in $\mathcal F$.
\end{theorem}
\begin{proof} Without loss of generality, we can assume that $k\ge(d+2)/2$ and that the $k$ subsets from $\mathcal F$ are $C_1,\ldots,C_k$. Therefore, we have to prove that there is $x\in M$ such that
	$$
	x\in \bigcap\limits_{i=1}^k{C_i} \; \mbox{ and } \; A(x)\in \bigcap\limits_{i=k+1}^{d+2}{C_i}
	$$
	
Set $C_{-i}:=A(C_i)$. Let $m:=\lceil{d/2}\rceil,$
$$
B_1:=C_1\cap(C_{-2}\cup\ldots\cup C_{-(m+1)}\cup C_{-(d+2)}),
$$
$$
B_2:=C_2\cap(C_{-3}\cup\ldots\cup C_{-(m+2)}\cup C_{-(d+2)}),
$$
$$
\vdots
$$
$$
B_d:=C_d\cap(C_{-(d+1)}\cup C_{-1}\cup\ldots\cup C_{-(m-1)}\cup  C_{-(d+2)}),
$$
$$
B_{d+1}:=C_{d+1}\cap(C_{-1}\cup\ldots\cup C_{-m}\cup C_{-(d+2)}).
$$

If $B_{-i}:=A(B_i)$, then
$$
\bigcup\limits_{i=1}^{d+1}{B_i\cup B_{-i}}=\bigcup\limits_{i=1}^{d+2}{C_i\cap(C_1\cup\ldots\cup C_{d+2})}=\bigcup\limits_{i=1}^{d+2}{C_i\cap M}=\bigcup\limits_{i=1}^{d+2}{C_i}=M.
$$
On the other hand, $B_i\subset C_i$ and $B_{-i}\subset C_{-i}$, hence $B_i\cap B_{-i}=\emptyset$. Therefore, the family of subsets $\{B_i\}$ satisfies the assumptions of Theorem \ref{t41}. It follows that
$$
Q:=B_1\cap\ldots\cap B_k\cap B_{-(k+1)}\cap\ldots\cap B_{-(d+1)}\ne\emptyset.
$$
Let $x\in Q$. Then $$x\in C_1\cap\ldots\cap C_k \, \mbox{ and } \, A(x)\in C_{-(k+1)}\cap\ldots\cap C_{-(d+1)}.$$
Since $k\ge m+1$ and $x\in B_1=C_1\cap(C_{-2}\cup\ldots\cup C_{-(m+1)}\cup C_{-(d+2)})$, we have $x\in C_{-(d+2)}$, i.e. $A(x)\in C_{d+2}$.
\end{proof}

\begin{cor} Let $(M,A)$ be a $d$-dimensional BUT--manifold. Then $M$ cannot
be covered by $d+1$ closed sets, none containing a pair $(x,A(x))$ of antipodal points.
\end{cor}

Note that the case $M={\Bbb S}^d$ was first considered by Lusternik and Schnirelmann in 1930.

\begin{proof} Suppose the converse, so $M$ can be covered by closed subsets $C_1,\ldots,C_{d+1}$.
	
	Let $C_{d+2}:=C_1$. Then this covering satisfies the assumptions of Theorem \ref{t42}. So there is $x$ such that
	$$
	x\in \bigcap\limits_{i=1}^{d+1}{C_i} \; \mbox{ and } \; A(x)\in C_{d+2}, \; \mbox{ i.e. } \; (x,A(x)) \in C_1,
	$$
a contradiction. 	
\end{proof}

\medskip

\medskip

\noindent{\bf Acknowledgment.} I  wish to thank Fr\'ed\'eric Meunier  for helpful discussions and  comments.

 \medskip

\noindent O. R. Musin\\ 
 University of Texas Rio Grande Valley, School of Mathematical and Statistical Sciences, One West University Boulevard, Brownsville, TX, 78520 \\
{\it E-mail address:} oleg.musin@utrgv.edu

\end{document}